\documentclass[11pt,letterpaper]{amsart}

\usepackage{amsmath,amsthm, latexsym, amssymb,mathrsfs,  tikz-cd}
\usepackage{stackengine,scalerel}
\usepackage{quiver}
\usepackage[colorlinks]{hyperref}
\usepackage[capitalize]{cleveref}
\definecolor{dark-red}{rgb}{0.6,0,0}
\definecolor{dark-green}{rgb}{0,0.4,0}
\definecolor{medium-blue}{rgb}{0,0,0.5}
\hypersetup{
    colorlinks, linkcolor={dark-red},
    citecolor={dark-green}, urlcolor={medium-blue}
}

\newcommand{\colim}{\mathrm{colim}}

\newcommand{\mc}[1]{\mathcal{#1}}
\newcommand{\mbb}[1]{\mathbb{#1}}
\newcommand{\mr}[1]{\mathrm{#1}}
\newcommand{\mf}[1]{\mathfrak{#1}}

\newcommand{\ur}{\mathrm{ur}}
\newcommand{\et}{\mathrm{\acute{e}t}}

\newcommand{\dR}{\mathrm{dR}}

\newcommand{\qlog}{\mathrm{qlog}} 
\newcommand{\can}{\mathrm{can}}
\DeclareMathOperator{\Lie}{Lie}
\DeclareMathOperator{\Spec}{Spec}
\DeclareMathOperator{\Spf}{Spf}
\DeclareMathOperator{\Gal}{Gal}

\DeclareMathOperator{\Hom}{Hom}
\DeclareMathOperator{\RigExt}{RigExt}
\DeclareMathOperator{\Red}{Red}

\newcommand{\pdiv}{p\mathrm{-div}}
\newcommand{\crys}{\mr{crys}}
\newcommand{\Ext}{\mr{Ext}}

\numberwithin{equation}{subsection}
\theoremstyle{plain}

\newtheorem*{theorem*}{Theorem}

\newtheorem{theorem}[subsection]{Theorem}
\newtheorem{corollary}[subsection]{Corollary}

\newtheorem{proposition}[subsection]{Proposition}
\newtheorem{lemma}[subsection]{Lemma}

\theoremstyle{definition}

\newtheorem{example}[subsection]{Example}
\newtheorem{definition}[subsection]{Definition}
\newtheorem{remark}[subsection]{Remark}
\newtheorem{question}[subsection]{Question}

\theoremstyle{remark}

\begin{document}

\title{The conjugate uniformization via 1-motives}

\begin{abstract} 
We use the $p$-divisible group attached to a 1-motive to generalize the conjugate $p$-adic uniformization of Iovita--Morrow--Zaharescu to arbitrary  $p$-adic formal semi-abelian schemes or $p$-divisible groups over the ring of integers in a $p$-adic field. This mirrors a mixed Hodge theory construction of the inverse uniformization map for complex semi-abelian varieties. We also highlight the geometric structure of the target of the conjugate uniformization map, which is an \'{e}tale cover of a negative Banach--Colmez space in the sense of Fargues--Scholze.  
\end{abstract}

\author{Sean Howe}
\address{Department of Mathematics, University of Utah, Salt Lake City UT 84112}
 \email{sean.howe@utah.edu}
 \author{Jackson S. Morrow}
\address{Department of Mathematics, University of California, Berkeley, Berkeley CA 94720}
 \email{jacksonmorrow@berkeley.edu}
 \author{Peter Wear}
\address{Department of Mathematics, University of Utah, Salt Lake City UT 84112}
 \email{peter.wear@utah.edu}
 \thanks{During the preparation of this article, J.S.M was partially supported by NSF RTG grant DMS-1646385 and later supported by NSF MSPRF grant DMS-2202960. P.W. was supported by NSF RTG grant DMS-1840190.}

\maketitle 
\tableofcontents

%\sean{Good! In particular re:references to IMZ I meant to say --- I didn't reference anything while I was writing this up but obviously there should be more; I wasn't sure which arguments are the same, you'll know better} 

\section{Introduction, statement of results, and context}

Let $K$ be a $p$-adic field; that is, a complete discretely valued extension of $\mbb{Q}_p$ with perfect residue field $\kappa$. Let $\overline{K}$ be an algebraic closure of $K$ and let $C$ be the completion of $K$. We write $K^\ur \subset \overline{K}$ for the maximal unramified extension 
and $\overline{\kappa}$ for the common residue field of $K^\ur$, $\overline{K}$, and $C$. For $G$ a $p$-divisible group or $p$-adic formal semi-abelian scheme over $\mc{O}_K$, we may evaluate points on a $p$-adically complete  $\mc{O}_K$-algebra $R$ by $G(R):=\lim_k G(R/p^k)$ (for a $p$-divisible group, the points \`{a} la Tate \cite[\S2.4]{tate:p-divisible-groups}). We write $G[p^\infty]$ for the $p$-divisible group, $T_p G = \lim_n G[p^n]$ for the Tate module, and $V_p G=T_p G \otimes \mbb{Q}_p$; the latter two are functors on $\mc{O}_K$-algebras (since even in the formal case $G[p^n]$ promotes automatically to a scheme over $\mc{O}_K$), but if we do not include an argument they are evaluated on $\overline{K}$ and equipped with the action of $\Gal(\overline{K}/K)$. 

For $G$ as above,  we have the functorial Hodge--Tate exact sequence \cite[\S4]{tate:p-divisible-groups} 
\[ 0 \rightarrow \Lie G\otimes_{\mc{O}_K} C(1) \rightarrow T_p G \otimes_{\mbb{Z}_p}C \rightarrow  \omega_{G[p^\infty] ^\vee} \otimes_{\mc{O}_K} C \rightarrow 0. \]
It depends only on $G[p^\infty]$, and by \cite[Corollary 2]{tate:p-divisible-groups} there is a unique Galois equivariant splitting 
\begin{equation}\label{eq.tate-splitting} T_p G \rightarrow \Lie G \otimes_{\mc{O}_K} C(1). \end{equation}
As in \cite[Theorem A]{imz}, the kernel of \cref{eq.tate-splitting} is $T_p G(\mc{O}_{K^\ur})$, i.e.~the Tate module of the maximal \'{e}tale $p$-divisible subgroup of $G$:~this is evidently contained in the kernel by functoriality, but the kernel is crystalline as a sub-representation of $T_p G$ and of Hodge--Tate weight zero by definition, thus unramified by a standard result in $p$-adic Hodge theory\footnote{To wit, if the Hodge--Tate weights are zero then the filtration on the associated filtered isocrystal is trivial, thus it corresponds to an unramified Galois representation under Fontaine's equivalence. See \cite[Appendix A]{imz} for another argument.}.

We will usually assume that $T_p G(K^\ur)=0$, equivalently that $G$ does not contain a non-trivial \'{e}tale $p$-divisible subgroup, equivalently that \cref{eq.tate-splitting} is injective. This condition is invariant under extension of $p$-adic fields and is not too serious, as, e.g., in the $p$-divisible case one can just quotient by the maximal \'{e}tale $p$-divisible subgroup to reduce to this setting. However, it simplifies statements considerably (see \cref{lemma.p-div-field-of-def} for why!). 

In this work, we construct an integration homomorphism  
\[ \overline{I}\colon G(\mc{O}_{K}) \rightarrow (\Lie G \otimes_{\mc{O}_K} C(1))/ T_p G, \]
by imitating with $p$-divisible groups a construction of the inverse of the uniformization map for a complex semi-abelian variety via the extensions of mixed Hodge structures attached to 1-motives. The map $\overline{I}$ is functorial and compatible with extensions $K \subseteq K'\subseteq C$ of $p$-adic fields. Our main result is
\begin{theorem}\label{theorem.main} Let $K$ be a $p$-adic field and let $G$ be a $p$-divisible group or $p$-adic formal semi-abelian scheme over $\mc{O}_K$. If $T_pG(K^\ur)=0$, then
\begin{enumerate}
\item  $\mr{Ker}(\overline{I})=G(\mc{O}_K)^{\pdiv}$, the subgroup of $p$-divisible elements (see \cref{defn:pdiv}). 
\item $\overline{I}(G[p^\infty](\mc{O}_K))$ is the set of $y \in (\Lie G \otimes_{\mc{O}_K} C(1))/ T_p G$ such that 
\begin{enumerate}
\item $y$ is stabilized by $\Gal(\overline{K}/K)$, 
\end{enumerate}
and, for $V_y := T_y \otimes \mbb{Q}_p$ where $T_y$ is the extension of $\mbb{Z}_p$ by $T_p G$ obtained by pulling back along $\mbb{Z}_p \rightarrow (\Lie G \otimes_{\mc{O}_K} C(1))/T_p G$, $1 \mapsto y$, 
\begin{enumerate}
\item[(b)] $V_y$ is a crystalline representation of $\Gal(\overline{K}/K)$, and
\item[(c)] the maximal unramified quotient of $V_y$ is a split extension of $\mbb{Q}_p$ by $V_p G(\overline{\kappa})$ as a $\Gal(\overline{\kappa}/\kappa)$-representation. 
\end{enumerate}
\end{enumerate}
\end{theorem}

We make some clarifying remarks on the statement of \cref{theorem.main}.
\begin{remark}\label{remark.clarifying-statements}
\begin{enumerate}
\item []
\item Because $T_p G(\mc{O}_K)=0$, any $p$-divisible 
$\mc{O}_K$-point is uniquely $p$-divisible, and the Fontaine construction (\cref{example.fontaine-p-adic}) induces 
\[ \Hom(\mbb{Z}[1/p], G(\kappa))=\Hom(\mbb{Z}[1/p], G(\mc{O}_K))=G(\mc{O}_K)^{\pdiv}. \]
In particular, the $p$-divisible elements are insensitive to purely ramified extension. Note that the prime-to-$p$ torsion, which is in bijection by reduction with prime-to-$p$ torsion in $G(\kappa)$, is always $p$-divisible. From this it also follows that if $G(\kappa)/G[p^\infty](\kappa)$ admits a set of $p$-divisible representatives in $G(\kappa)$ then $\overline{I}(G[p^\infty](\mc{O}_{K}))=\overline{I}(G(\mc{O}_K))$ --- this occurs, e.g., when $\kappa$ is algebraically closed or finite. 
\item The crystalline condition (b) is necessary --- for example, when $G=\mu_{p^\infty}$, the non-crystalline Tate module of the semistable elliptic curve $\mbb{G}_m^{\mr{an}}/p^\mbb{Z}$ can be obtained from a $y$ satisfying (a) and trivially (c).  
\item The condition (c) is automatically verified for $G$ such that  \[ \mr{Ext}_{\Gal(\overline{\kappa}/\kappa)}(\mbb{Q}_p, V_p G(\overline{\kappa}))=0.\] For example, it suffices that $\kappa$ be algebraically closed, or that $G[p^\infty]_\kappa$ be connected, or that $G_\kappa$ be a semi-abelian scheme and $\kappa$ be finite (see \cref{corollary.imz} below). 
 \end{enumerate}
\end{remark}

\begin{corollary}\label{corollary.imz} 
Suppose $[K:\mbb{Q}_p]<\infty$,  
$G$ is a $p$-adic formal semi-abelian scheme over $\mc{O}_K$, and $T_p G(\mc{O}_{K^\ur})=0$. Then $\overline{I}$ factors as the projection
\begin{equation}\label{eq.cor-tors-decomp} G(\mc{O}_{K})=G[p^\infty](\mc{O}_{K}) \times G(\mc{O}_{K})^{\textrm{prime-to-$p$ torsion}} \rightarrow  G[p^\infty](\mc{O}_{K}) \end{equation}
composed with the injection $\overline{I}\colon G[p^\infty](\mc{O}_{K}) \hookrightarrow (\Lie G \otimes_{\mc{O}_K} C(1)) / T_p G$. The image consists of the points $y$ stabilized by $\Gal(\overline{K}/K)$ with $V_y$  crystalline (notation as in the statement of \cref{theorem.main}). 
\end{corollary}
\begin{proof}
Note the decomposition in \cref{eq.cor-tors-decomp} is immediate from the following properties of $G$ and the reduction map $\Red\colon G(\mc{O}_K)\rightarrow G(\kappa)$: 
\begin{enumerate}
\item $G(\kappa)$ is a finite abelian group.
\item The map $\Red$ induces an isomorphism
 \[ G(\mc{O}_{K})^{\textrm{prime-to-$p$ torsion}}\cong G(\kappa)^{\textrm{prime-to-$p$ torsion}}\]
\item $G[p^\infty](\mc{O}_{K}) = \Red^{-1}(G[p^\infty](\kappa))=\Red^{-1}(G(\kappa)[p^\infty]).$	
\end{enumerate}
Now, as noted in \cref{remark.clarifying-statements}-(1), reduction induces  
\[ G(\mc{O}_K)^{\pdiv}=\Hom(\mbb{Z}[1/p], G(\mc{O}_{K}))=\Hom(\mbb{Z}[1/p], G(\kappa)). \]
Since $G(\kappa)$ is a finite abelian group, the right-hand-side is identified by evaluation at $1$ with the prime-to-$p$ torsion in $G(\kappa)$, thus the left-hand side is also prime-to-$p$ torsion. Since every prime-to-$p$ torsion point is $p$-divisible,
\[  G(\mc{O}_K)^{\pdiv}=G(\mc{O}_{K})^{\textrm{prime-to-$p$ torsion}}. \]
It remains only to show that in this case condition (c) in the characterization of the image in \cref{theorem.main} is superfluous. But this follows since here the Galois representation is determined by a single matrix corresponding to Frobenius, but the sub-representation $V_p G(\kappa)$ does not have $1$ as an eigenvalue (since otherwise there would be a non-trivial fixed vector giving rise to infinitely many points in $G(\kappa)$), thus the extension is split. 
\end{proof}

We show in \cref{s.integration-universal-cover-other-constructions} that $\overline{I}$ can also be constructed using Fontaine integration \cite{fontaine:integral}, so that it agrees with the map studied in \cite{imz}. \cref{corollary.imz} thus generalizes the main result of \cite{imz}, which treats the case where $G$ is the base change to $\mc{O}_K$ of a good reduction abelian variety over a finite unramified extension of $\mbb{Q}_p$.  
Indeed, the main motivation for this work was to give a simpler and more general construction of the uniformization of loc.~cit.~rendering its key properties evident via the theory of $p$-divisible groups.  Beyond extending the field of definition, our generality allows for, e.g., 
\begin{enumerate}
\item the connected component of the N\'{e}ron model of an arbitrary abelian variety (after finite extension to obtain semi-abelian reduction), and
\item non-algebraizable good reduction abeloid varieties.  
\end{enumerate}

We refer to the characterization of the $\mc{O}_{K}$-points of certain $p$-divisible groups described in \cref{corollary.imz} and \cref{remark.p-div-unif} as the \emph{conjugate uniformization} because the splitting of the Hodge--Tate filtration is an analog of the conjugate filtration in complex Hodge theory (and to distinguish it from the Scholze--Weinstein uniformization, recalled below in \cref{s.geometry}). 

\begin{remark}\label{remark.p-div-unif} 
In \cite{imz}, the trivial decomposition of \cref{eq.cor-tors-decomp} into the points of the $p$-divisible group and the prime-to-$p$ torsion is not clearly described, so that one factor is never identified as being the points of the $p$-divisible group. As a result, in \cite{imz}, this uniformization is presented as a result about abelian varieties. By contrast, we wish to emphasize that this uniformization is purely a result about $p$-divisible groups. The identification of $G[p^\infty](\mc{O}_{K})$ with the set of crystalline points in $(\Lie G \otimes_{\mc{O}_K} C(1)) / T_p G$ is formulated entirely in the world of $p$-divisible groups over $\mathcal{O}_K$ --- the source, the map $\overline{I}|_{G[p^\infty](\mc{O}_{K})}$, and the target all only depend on $G[p^\infty]$. Indeed, the only place we use the semi-abelian scheme in the proof of \cref{corollary.imz} is to conclude that $G[p^\infty](\kappa)$ has finitely many points. Thus this part of the corollary holds for any $p$-divisible group $H/\mc{O}_K$ with the same properties, which can be formulated by requiring that $T_p H(K^\ur)=0$ and $T_p H(\kappa)= 0$. 
\end{remark}

\begin{remark}\label{remark.intro-continuity}
In \cite{imz} an emphasis is placed on the continuity properties of $\overline{I}$ on $G(\mc{O}_{\overline{K}})$ --- this aspect did not appear in \cref{theorem.main} or \cref{corollary.imz} because it plays no role in the construction and because we have stated these results using points in a single $p$-adic field while the subtle topological issue in question apears only if we consider all $\mc{O}_{\overline{K}}$-points at once. See \cref{remark:continuity} for further discussion.
\end{remark}

Before describing our construction of $\overline{I}$ and explaining how it leads to a proof of \cref{theorem.main}, we take a brief detour to explain the analogous construction of the inverse uniformization map for complex abelian varieties.

\subsection{Uniformization of complex semi-abelian varieties}
\label{ss.uniformization}
Let $G$ be a semi-abelian variety over $\mbb{C}$, that is, an extension of an abelian variety by a torus. Via exponentiation, we obtain a uniformization 
\[ \exp\colon \Lie G/ H_1(G(\mbb{C}), \mbb{Z}) \rightarrow G(\mbb{C}). \]
The inverse can be constructed by integration:~if we identify $\Lie G$ with the dual to the space of invariant differentials on $G$, then $x \in G(\mbb{C})$ maps to
\[ \overline{I}_{\mbb{C}}(x) \colon \omega \mapsto \int_e^x \omega. \]

This integral factors through the category of $1$-motives \cite[\S10]{deligne:hodge-III} --- that is, to the point $x$ we can associate the $1$-motive $G_x\colon \mbb{Z}\rightarrow G$ where the map sends $1$ to $x$. Taking homology gives an extension of mixed Hodge structures
\[ 0 \rightarrow H_1(G, \mbb{Z}) \rightarrow H_1(G_x, \mbb{Z}) \rightarrow \mbb{Z} \rightarrow 0.\]
Concretely, $H_1(G_x, \mbb{Z})$ is the homology of $G$ relative to $\{e, x\}$, so that the class $1$ in the quotient trivial mixed Hodge structure $\mbb{Z}$ corresponds to any path from $e$ to $x$. The integration map can then be obtained by quotienting by the Hodge filtration $\mr{Fil}^{-1}$ in the first two terms to obtain 
% https://q.uiver.app/?q=WzAsMTAsWzAsMCwiMCJdLFsxLDAsIkhfMShHLCBcXG1iYntafSkiXSxbMiwwLCJIXzEoRydfcCwgXFxtYmJ7Wn0pIl0sWzMsMCwiXFxtYmJ7Wn0iXSxbNCwwLCIwIl0sWzAsMSwiMCJdLFsxLDEsIlxcTGllIEciXSxbMiwxLCJcXExpZSBHIl0sWzMsMSwiXFxMaWUgRyAvIEhfMShHLFxcbWJie1p9KSJdLFs0LDEsIjAiXSxbMyw4XSxbMSw2XSxbMiw3XSxbNSw2XSxbMCwxXSxbMSwyXSxbNiw3XSxbMiwzXSxbNyw4XSxbMyw0XSxbOCw5XV0=
\[\begin{tikzcd}
	0 & {H_1(G, \mbb{Z})} & {H_1(G_x, \mbb{Z})} & {\mbb{Z}} & 0 \\
	0 & {\Lie G} & {\Lie G_x} & {\Lie G / H_1(G,\mbb{Z})} 
	\arrow[from=1-4, to=2-4]
	\arrow[from=1-2, to=2-2]
	\arrow[from=1-3, to=2-3]
	\arrow[from=2-1, to=2-2]
	\arrow[from=1-1, to=1-2]
	\arrow[from=1-2, to=1-3]
	\arrow[from=2-2, to=2-3, equal]
	\arrow[from=1-3, to=1-4]
	\arrow[from=2-3, to=2-4]
	\arrow[from=1-4, to=1-5]
\end{tikzcd}\]
Tracing through the construction, we find $\overline{I}_{\mbb{C}}(x)$ is the image of $1 \in \mbb{Z}$.

\begin{example}\label{example.log-gm-complex} When $G=\mbb{G}_m$, if we trivialize $\Lie G$ via the vector field $x \partial_x$, then the uniformization can be identified with the complex exponential
\[ \exp\colon \mbb{C}/\mbb{Z}(1) \rightarrow \mbb{C}^\times \]
whose inverse is the logarithm map obtained by integrating the dual basis $\frac{dx}{x}$ for $\omega_{\mbb{C}^\times}$ from $1$ to $x$. It is a linear algebra exercise to see that there is a natural identification $\mr{Ext}_{\mbb{Z}-\mr{MHS}}(\mbb{Z}, \mbb{Z}(1)) = \mbb{C}/\mbb{Z}(1)$, and to verify that the mixed Hodge structure on the first homology of the (twice punctured) nodal cubic obtained from $\mbb{C}^\times$ by glueing $1$ and $x$, which is canonically an extension of $\mbb{Z}$ by $\mbb{Z}(1)$, is matched in this identification with $\log (x)$. 
\end{example}
  
\subsection{Construction of the map $\overline{I}$ and outline of proof}

Returning to setup at the start of the introduction, we now construct a homomorphism 
\[ \overline{I}\colon G(\mc{O}_{K}) \rightarrow (\Lie G \otimes_{\mathcal{O}_K} C(1))/T_p G. \]
The construction is functorial and compatible with extension $K \subset K' \subset C$. To construct it, to any point $x \in G(\mc{O}_K)$ we attach the Kummer extension 
\[ \mc{E}_x\colon 0 \rightarrow G[p^\infty] \rightarrow G_x[p^\infty] \rightarrow \mbb{Q}_p/\mbb{Z}_p \rightarrow 0\]
of $p$-divisible groups over $\mc{O}_K$ given by formally adjoining $p$-power roots of $-x$ to $G$. When $G$ is a $p$-adic formal semi-abelian scheme, this can be identified with the $p$-divisible group of the $p$-adic formal $1$-motive $\mbb{Z} \xrightarrow{1\mapsto x} G$ with the extension structure coming from the weight filtration.

We then take Tate modules and apply \cref{eq.tate-splitting} to obtain
% https://q.uiver.app/?q=WzAsMTAsWzAsMCwiMCJdLFsxLDAsIlRfcCBHIl0sWzIsMCwiVF9wIEdfeCJdLFszLDAsIlxcbWJie1p9X3AiXSxbNCwwLCIwIl0sWzAsMSwiMCJdLFsxLDEsIlxcTGllIEcoMSkiXSxbMiwxLCJcXExpZSBHX3goMSkiXSxbMywxLCJcXExpZSBHKDEpIC8gVF9wIEciXSxbNCwxLCIwIl0sWzMsOF0sWzEsNl0sWzIsN10sWzUsNl0sWzAsMV0sWzEsMl0sWzYsNywiPSIsMSx7InN0eWxlIjp7ImJvZHkiOnsibmFtZSI6Im5vbmUifSwiaGVhZCI6eyJuYW1lIjoibm9uZSJ9fX1dLFsyLDNdLFs3LDhdLFszLDRdLFs4LDldXQ==
\[\begin{tikzcd}[column sep = 1em]
	0 & {T_p G} & {T_p G_x} & {\mbb{Z}_p} & 0 \\
	0 & {\Lie G \otimes_{\mc{O}_K} C(1)} & {\Lie G_x \otimes_{\mc{O}_K}C(1)} & {(\Lie G \otimes_{\mc{O}_K} C(1)) / T_p G} 
	\arrow[from=1-4, to=2-4]
	\arrow[from=1-2, to=2-2]
	\arrow[from=1-3, to=2-3]
	\arrow[from=2-1, to=2-2]
	\arrow[from=1-1, to=1-2]
	\arrow[from=1-2, to=1-3]
	\arrow[from=2-2, to=2-3, equal]
	\arrow[from=1-3, to=1-4]
	\arrow[from=2-3, to=2-4]
	\arrow[from=1-4, to=1-5]
\end{tikzcd}\]
We define $\overline{I}(x)$ to be the image of $1 \in \mbb{Z}_p$ in ${(\Lie G \otimes_{\mc{O}_K} C(1)) / T_p G} $. 

To prove \cref{theorem.main}, we first establish that all \emph{rigidified} extensions (\cref{defn:rigidified}) of $\mbb{Q}_p/\mbb{Z}_p$ by $G[p^\infty]$ are of the form $\mc{E}_x$ for $x \in G[p^\infty](\mc{O}_K)$ (cf.~\cref{theorem.rig-ext-iso}). Using that $\mc{E}_x$ is split if and only if $x$ is $p$-divisible and a simple argument again using that crystalline representations of Hodge--Tate weight zero are unramified, we can characterize the kernel on all of $G(\mc{O}_K)$ as the $p$-divisible elements. The crystalline characterization of the image is then immediate from the equivalence between lattices in crystalline representations with Hodge--Tate weights $\{0,1\}$ and $p$-divisible groups (see e.g., \cite{kisin:cyrstallinereps} or \cite[Corollary 6.2.3]{scholze-weinstein:p-div}) --- indeed, this gives an extension of $\mbb{Q}_p/\mbb{Z}_p$ by $G[p^\infty]$ attached to any $y$ satisfying (a) and (b), and the analysis of extensions indicated already shows it comes from an $x\in G[p^\infty](\mc{O}_K)$ if and only if it also satisfies (c).

\subsection{Outline} In \cref{s.preliminaries}, we recall some constructions regarding $p$-divisible groups, $p$-adic formal semi-abelian schemes, universal covers, and $p$-divisible elements. In \cref{s.p-div-1-mot}, we give a construction of the $p$-divisible group attached to a $1$-motive, and in \cref{s.rig-ext}, we discuss rigidified extensions of $p$-divisible groups. Because it is of independent interest and the proofs are not any more difficult, in \cref{s.p-div-1-mot} and \cref{s.rig-ext}, we work in more generality than will be needed for the application to \cref{theorem.main} and also discuss some complements (e.g., the $p$-divisible group of a Raynaud-uniformized abeloid variety) --- for the proof of \cref{theorem.main}, the key result is the identification of a $p$-divisible group $G$ with the moduli of rigidified extensions of $\mbb{Q}_p/\mbb{Z}_p$ by $G$. In \cref{s.proof}, we prove \cref{theorem.main}, and in \cref{s.integration-universal-cover-other-constructions}, we show the equivalence with the definition of $\overline{I}$ via the Fontaine integral and give another construction using the crystalline incarnation of the universal cover. In \cref{s.geometry}, we discuss geometric aspects of the conjugate uniformization.

 \section{Preliminaries}\label{s.preliminaries}
 In this section, we recall the construction of the universal cover of a $p$-divisible group or $p$-adic formal semi-abelian  scheme and its basic properties.
 
 Let $R$ be a $p$-adically complete ring, and let $G$ be a $p$-divisible group or a $p$-adic formal semi-abelian scheme over $R$. The \textsf{universal cover of $G$} is the functor on $p$-adically complete $R$-algebras
 \[ \tilde{G}=\Hom(\mbb{Z}[1/p], G) = \lim (G \xleftarrow{p} G \xleftarrow{p} \cdots ), \; S \mapsto \Hom_{\mbb{Z}}(\mbb{Z}[1/p], G(S)) \]
 When $G$ is a $p$-divisible group, we can replace $\Hom_{\mbb{Z}}(\mbb{Z}[1/p], \cdot)$ with $\Hom_{\mbb{Z}_p}(\mbb{Q}_p, \cdot)$. 
 
 Inside of $\tilde{G}$, we have the Tate module
 \[ T_p G = \Hom( \mbb{Z}[1/p]/\mbb{Z}, G)=\Hom( \mbb{Q}_p/\mbb{Z}_p, G) = \lim(1\xleftarrow{p} G[p]\xleftarrow{p} G[p^2]\xleftarrow{p} \cdots) \]
 Of course, $T_p G$ only depends on $G[p^\infty]$, and it is the kernel of the projection $\tilde{G} \rightarrow G$ given by evaluation at $1$ (an fpqc surjection).  
 
 The key property of $\tilde{G}$ is that, by a construction due to Fontaine, it is invariant under topologically nilpotent thickenings:
 \begin{proposition}\label{prop.fontaine-construction}
 Let $S$ be a $p$-adically complete $R$-algebra and let $I$ be a topologically nilpotent ideal in $S$. Then reduction induces an isomorphism
 \[ \tilde{G}(S)=\tilde{G}(S/I) \]	
 \end{proposition}
 \begin{proof}
By considering liftings from $S/(I, p^k)$ to $S/I$ and passing to the limit, it suffices to suppose $p^k=0$ on $S$ and that $I$ is nilpotent. In this case, let $(g_1, g_2, \ldots) \in \tilde{G}(S/I)$, and choose arbitrary element-wise lifts $\tilde{g}_1, \tilde{g}_2, \ldots $ to $G(S)$. Then, we claim that for $N$ sufficiently large, 
 \[ (\tilde{g}_{N}^{p^N}, \tilde{g}_{N+1}^{p^N}, \ldots) \]
  is independent of choices and furnishes a lift to $\tilde{G}(S)$ that is independent of the choices thus unique. Indeed, two different lifts of $g_i$ differ by an element of $\ker (G(S) \rightarrow G(S/I))$, and the Drinfeld construction \cite[Lemma 1.1.2]{katz:serre-tate} shows a large power of $p$ annihilates this subgroup. \end{proof}
  
  \begin{remark}
  We note that Proposition \ref{prop.fontaine-construction} for $G$ a $p$-divisible group appears as \cite[Proposition 3.1.3(ii)]{scholze-weinstein:p-div}, and the proofs are similar. 
  \end{remark}

\begin{definition}\label{defn:pdiv}
We say an element of $G(S)$ is (uniquely) \textsf{$p$-divisible} if it admits a (unique) compatible system of $p$-power roots of unity in $G(S)$.
\end{definition}

Note that an element of $G(S)$ is $p$-divisible if and only if it is in the image of $\tilde{G}(S) \rightarrow G(S)$, and that the $p$-divisible elements are uniquely $p$-divisible if and only if $T_pG(S)=0$.

\begin{example}\label{example.fontaine-p-adic} If $K$ is a $p$-adic field with residue field $\kappa$ and $G/\mc{O}_K$, then \cref{prop.fontaine-construction} gives
 \[ \tilde{G}(\mc{O}_K)= \tilde{G}(\kappa). \]
In particular, $\tilde{G}(\mc{O}_K)$ and the $p$-divisible elements of $G(\mc{O}_K)$ are invariant under finite totally ramified extensions. Note that a $p$-divisible element lies in the formal neighborhood of the identity if and only if it corresponds to an element of $T_p G(\kappa)$, while it lies in the $p$-divisible group if and only if it corresponds to an element of $V_p G(\kappa)=T_p G(\kappa) \otimes \mbb{Q}_p$.

Note that $T_p G(\kappa)$ is often much larger than $T_p G(\mc{O}_{K})$ --- for example, if $\kappa$ is algebraically closed, then $G_\kappa[p^\infty] \cong G^\wedge_\kappa \times (\mbb{Q}_p/\mbb{Z}_p)^r$ where $G^\wedge_\kappa$ is the formal group/connected component of the identity and $r$ is the rank of $T_p G(\kappa)$, but for a generic lift of $G_{\kappa}$ to $\mc{O}_K$ no copy of $\mbb{Q}_p/\mbb{Z}_p$ will lift to a subgroup so we will have $T_p G(\mc{O}_K)=0$. In this case, $T_pG(\kappa)$ is identified with the subgroup of $G^\wedge(\mc{O}_K)$ consisting of points that are (uniquely) $p$-divisible in $G(\mc{O}_K)$, and in fact this subgroup determines the extension structure for the connected-\'{e}tale sequence (see \cref{example.connected-etale}). 
\end{example}

\begin{remark}\label{remark.isogeny-universal-cover} It also follows from \cref{prop.fontaine-construction} that, for $H_1$ and $H_2$ $p$-divisible groups over $R$ and $I$ an open topologically nilpotent ideal of $R$, 
\[ \Hom(\widetilde{H}_1, \widetilde{H}_2)=\Hom(H_{1,R/I},H_{2,R/I}) \otimes \mbb{Q}_p. \] 
In other words, to give a map of universal covers is the same as to give a map in the isogeny-category of $p$-divisible groups over $R/I$. This is a helpful way to encode the independence of the latter on the choice of $I$ while also bringing to the forefront the Fontaine lifts as in the proof of \cref{prop.fontaine-construction}. We use this only in the case when $H_1$ is \'{e}tale, which immediately reduces to the case that $H_1=\mbb{Q}_p/\mbb{Z}_p$ where it is an immediate consequence of \cref{prop.fontaine-construction} and the identity $T_p H_2(A)\otimes \mbb{Q}_p = \tilde{H}_2(A)$ for any $R$-algebra $A$ such that $p$ is nilpotent in $A$. 
\end{remark}

\section{The $p$-divisible group of a 1-motive}\label{s.p-div-1-mot}
In this section, we construct the $p$-divisible group of a 1-motive. 
First, we recall the notion of a 1-motive (see \cite[\S10]{deligne:hodge-III}, \cite{barbieri-viale:1-motives}). Let $R$ be a ring. 

\begin{definition}
A \textsf{1-motive over $\Spec R$} is a map $\varphi\colon M \rightarrow G$ where $G$ is a semi-abelian scheme over $R$ and $M$ is an \'{e}tale $\mbb{Z}$-local system on $\Spec R$. 
\end{definition}

Given a $1$-motive $\varphi$ and a prime $p$, we can construct an extension of group schemes over $R$
\[ \mc{E}_\varphi\colon 0 \rightarrow G \rightarrow G_{\varphi} \rightarrow M \otimes (\mbb{Z}[1/p]/ \mbb{Z}) \rightarrow 0 \]
by formally adjoining $p$-power roots along $\varphi$. Precisely,  $G_\varphi$ is the push-out 
% https://q.uiver.app/?q=WzAsNCxbMCwwLCJcXHVse019Il0sWzEsMCwiRyJdLFswLDEsIlxcdWx7TSBcXG90aW1lcyBcXG1iYntafVsxL3BdfSJdLFsxLDEsIkdfXFx2YXJwaGkiXSxbMCwyXSxbMCwxLCJcXHZhcnBoaSIsMl0sWzIsM10sWzEsM11d
\begin{equation}\label{eq.push-out-diagram}\begin{tikzcd}
	{M} & G \\
	{M \otimes \mbb{Z}[1/p]} & {G_\varphi}
	\arrow[from=1-1, to=2-1]
	\arrow["-\varphi"', from=1-1, to=1-2]
	\arrow[from=2-1, to=2-2]
	\arrow[from=1-2, to=2-2]
\end{tikzcd}\end{equation}
which we can realize concretely as a disjoint union equipped with an addition law defined by carrying:~if we pass to a finite \'{e}tale cover and fix a trivialization $M\cong \mbb{Z}^m$, then, writing  $I_p=\mbb{Z}[1/p] \cap [0,1)  \subset \mbb{Q}$, 
\begin{equation}\label{eq.carrying-law} G_\varphi \cong \bigsqcup_{t\in I_p^m} G, \; g_{s} + h_t=(g+h - \varphi(\lfloor s+t \rfloor))_{\{s+t\}} \end{equation}
where here $\lfloor \cdot \rfloor$ denotes floor and $\{ \cdot \}$ denotes the fractional part. 

Note in particular that
\begin{equation}\label{eq.pntorsion} G_{\varphi}[p^n] = \bigsqcup_{t \in \left(1/p^n\mbb{Z} \cap [0,1)\right)^m } G_{\varphi}[p^n]_{t} \end{equation}
where $G_{\varphi}[p^n]_t$ is the fiber of multiplication by $p^n$ on $G$ above $\varphi(p^n t)$. 

In particular, $G_\varphi[p^\infty]$ is a $p$-divisible group, and  $\mc{E}_\varphi[p^\infty]$ is an extension of $p$-divisible groups
\[ \mathcal{E}_{\varphi}[p^{\infty}]\colon 0 \rightarrow G[p^\infty] \rightarrow G_{\varphi}[p^\infty] \rightarrow M \otimes (\mbb{Q}_p/\mbb{Z}_p) \rightarrow 0. \]

\begin{definition}
    The \textsf{$p$-divisible group attached to the $1$-motive $\varphi\colon M \rightarrow G$} is $G_{\varphi}[p^\infty]$, and the extension structure $\mc{E}_\varphi[p^\infty]$ is the weight filtration.
\end{definition}

\begin{remark}
Over $\mbb{C}$, \cref{eq.pntorsion} agrees with the $p^n$ torsion of a 1-motive as constructed in \cite[(10.1.5)]{deligne:hodge-III} (in particular, the minus sign in the push-out diagram \cref{eq.push-out-diagram} arises naturally from the Koszul sign rule). More generally, this construction agrees with the construction of the $p$-divisible group attached to a $1$-motive in \cite[\S2.4]{barbieri-viale:1-motives}.
\end{remark}

Next, we define a $p$-adic formal 1-motive.

\begin{definition}
If $R$ is $p$-adically complete, a \textsf{$p$-adic formal $1$-motive} is a map $M \rightarrow G$ where $M$ is an \'{e}tale $\mbb{Z}$-local system on $\Spf R$ (or just $\Spec R/p$) and $G$ is a $p$-adic formal semi-abelian scheme over $\Spf R$. 
\end{definition}

If $G$ is a $p$-adic formal semi-abelian scheme over $R$, then applying the previous construction over $R/p^n$ for all $n$ yields an extension of $p$-divisible groups $\mc{E}_{\varphi}[p^\infty]$ over $R$. Note that if $\varphi$ factors through $G[p^\infty]$, then in the construction of $G_\varphi[p^\infty]$ we may dispense with $G$ altogether and work from the beginning with $G[p^\infty]$ in its place. Note that the maps $\varphi\colon{\mbb{Z}} \rightarrow G[p^\infty]$ correspond exactly to the points \`{a} la Tate,
\[  G[p^\infty](R) := \lim_{k}\colim_n G[p^n](R/p^k), \]
which typically is much larger than $G(R)[p^\infty]$.

On the other hand, there are sometimes very interesting extensions of $G[p^\infty]$ that can only be seen by considering points in $G$, as the following example illustrates.
\begin{example}\label{example.tate-curve}
Let $E/\mbb{Z}(\!(q)\!)$ be the Tate elliptic curve. Then there is a canonical isomorphism $E[p^\infty] \cong (\mbb{G}_m)_{1\mapsto q}[p^\infty]$ where $1 \mapsto q$ denotes the map $\varphi\colon \mbb{Z}\rightarrow \mbb{G}_m(\mbb{Z}(\!(q)\!))$ sending $1$ to $q$. Even if we $p$-adically complete, this extension of $\mbb{Q}_p/\mbb{Z}_p$ by $\mu_{p^\infty}$ still does not arise from a $1$-motive factoring through $\mu_{p^\infty}$, which see only Serre--Tate extensions; this version was treated in detail in \cite{howe:circle-action}, where the construction of $\mc{E}_\varphi$ for $G=\mbb{G}_m$ and $M=\mbb{Z}$ was referred to as the theory of Kummer extensions and was used to unify computations for $p$-adic modular forms in Serre--Tate and cuspidal coordinates. 
\end{example}

\begin{example}
There is a completely analogous construction for rigid analytic 1-motives, and using this one can construct the $p$-divisible group of a Raynaud uniformized abeloid variety in the same way using a rigid 1-motive. In this case one obtains a $p$-divisible group over $K$ that does not extend to $\mc{O}_K$ (but can sometimes still be made sense of algebraically over a complete $\mc{O}_K$-algebra in the limit by adding a formal variable as in  \cref{example.tate-curve}). 
\end{example}

These kinds of examples will not play a serious role in the remainder of this work because of \cref{theorem.canonical-rigidifications} below, which implies that for $H$ a $p$-divisible group over the ring of integers in a $p$-adic field, up to a minor discrepancy all extensions of $\mbb{Q}_p/\mbb{Z}_p$ by $H$ can be obtained already from points of $H$. 

\section{Rigidifed Extensions}\label{s.rig-ext}
In this section, we define rigidifed extensions and prove \cref{theorem.canonical-rigidifications}. 
Let $R$ be a $p$-adically complete ring, and let $H/R$ be a $p$-divisible group. Suppose $M$ is a $\mbb{Z}_p$-local system on $\Spf R$ (equivalently $\Spec R/p$) and $\varphi\colon M \rightarrow H$ is a map. 
We note also that $M$ is equivalent to the \'{e}tale $p$-divisible group $M \otimes \mbb{Q}_p/\mbb{Z}_p$ (from which $M$ is recovered as the Tate module). 

\begin{remark}
It is tempting to call $\varphi$ a $p$-divisible $1$-motive, but this would be a mistake (see \cref{remark.rig-BKF}). 
\end{remark}

\begin{example} 
Suppose $\varphi\colon  M \rightarrow G$ is a $p$-adic formal 1-motive over $\Spf R$ such that $\varphi$ factors through $G[p^\infty]$. Then $\varphi$ extends uniquely to 
\[ \varphi \otimes \mbb{Z}_p \colon M \otimes \mbb{Z}_p \rightarrow G[p^\infty]. \]	
\end{example}

Given $\varphi\colon M \rightarrow H$, we form the pushout $H_\varphi$ analogous to  the earlier construction \cref{eq.push-out-diagram} with $1$-motives but replacing $\mbb{Z}[1/p]$ with $\mbb{Q}_p$ and $\mbb{Z}$ with $\mbb{Z}_p$:
\[
\begin{tikzcd}
M\arrow{r}{-\varphi} \arrow{d} & H \arrow{d}\\
M\otimes \mathbb{Q}_p \arrow{r} & H_{\varphi}.
\end{tikzcd}
\]
Note that $\mbb{Q}_p$ and $\mbb{Z}_p$ are equipped with their natural topologies and should be interpreted here as topological constant sheaves.

This admits an identical explicit description via a carrying law after pro-finite \'{e}tale cover to trivialize $M$. In particular, we obtain a short exact sequence
\[ \mc{E}_\varphi\colon 0 \rightarrow H \rightarrow H_\varphi \rightarrow M \otimes (\mbb{Q}_p/\mbb{Z}_p) \rightarrow 0 \]
and at the level of universal covers we have 
\[ \tilde{\mc{E}}_\varphi\colon 0 \rightarrow \tilde{H} \rightarrow \tilde{H}_\varphi \rightarrow M \otimes \mbb{Q}_p \rightarrow 0. \]
In fact, there is another important piece of data in the mix:~there is a canonical section $s_\varphi\colon M \otimes \mbb{Q}_p\rightarrow \tilde{H}_\varphi$ of the induced extension of universal covers $\tilde{\mc{E}}_\varphi$ coming from the canonical map $M \otimes \mbb{Q}_p \rightarrow H_\varphi$ extending $-\varphi$. 

\begin{definition}\label{defn:rigidified}
With the notation as above, we refer to an extension $\mc{E}_{\varphi}$ equipped with a section $s$ of $\tilde{\mc{E}}$ as \textsf{rigidified}, and let $\RigExt(M \otimes (\mbb{Q}_p/\mbb{Z}_p), H)$ refer to the functor on $p$-adically complete $R$-algebras sending $S$ to the set of isomorphism classes of rigidified extensions of $H_S$ by $M_S \otimes (\mbb{Q}_p/\mbb{Z}_p)$. 
\end{definition}

We now describe how one can interpret the functor $\RigExt(M \otimes (\mbb{Q}_p/\mbb{Z}_p), H)$ in terms of $p$-adic 1-motives. 
\begin{theorem}\label{theorem.rig-ext-iso}
The assignment $\varphi \mapsto (\mc{E}_\varphi, s_\varphi)$ is an isomorphism of functors from $p$-adically complete $R$-algebras to abelian groups 
\[ \Hom(M, H) \rightarrow \RigExt(M \otimes (\mbb{Q}_p/\mbb{Z}_p), H) \]
where the right-hand side is equipped with the Baer sum.\end{theorem}
\begin{proof}
If $(\mc{E}, s)$ is a rigidified extension, let $s_0$ denote the composition of $s$ with projection to the first coordinate from $\tilde{\mc{E}} \to \mc{E}$:
% https://q.uiver.app/?q=WzAsMTAsWzAsMCwiMCJdLFsxLDAsIlxcdGlsZGV7SH0iXSxbMiwwLCJcXHRpbGRle1xcbWF0aGNhbHtFfX1fcyJdLFszLDAsIk1cXG90aW1lcyBcXG1hdGhiYntRfV9wIl0sWzQsMCwiMCJdLFswLDEsIjAiXSxbMSwxLCJIIl0sWzIsMSwiXFxtYXRoY2Fse0V9X3MiXSxbMywxLCJNXFxvdGltZXMgXFxtYXRoYmIgUV9wL1xcbWF0aGJiIFpfcCJdLFs0LDEsIjAiXSxbMCwxXSxbMSwyXSxbMiwzLCIiLDAseyJvZmZzZXQiOjF9XSxbMywyLCJzIiwyLHsib2Zmc2V0IjoxfV0sWzMsNywic18wIiwyXSxbMSw2XSxbMiw3XSxbMyw4XSxbNSw2XSxbNiw3XSxbNyw4XSxbOCw5XSxbMyw0XV0=
\[\begin{tikzcd}
	0 & {\tilde{H}} & {\tilde{\mathcal{E}}} & {M\otimes \mathbb{Q}_p} & 0 \\
	0 & H & {\mathcal{E}} & {M\otimes (\mathbb{Q}_p/\mathbb{Z}_p)} & 0
	\arrow[from=1-1, to=1-2]
	\arrow[from=1-2, to=1-3]
	\arrow[shift right=1, from=1-3, to=1-4]
	\arrow["s"', shift right=1, from=1-4, to=1-3]
	\arrow["{s_0}"', from=1-4, to=2-3]
	\arrow[from=1-2, to=2-2]
	\arrow[from=1-3, to=2-3]
	\arrow[from=1-4, to=2-4]
	\arrow[from=2-1, to=2-2]
	\arrow[from=2-2, to=2-3]
	\arrow[from=2-3, to=2-4]
	\arrow[from=2-4, to=2-5]
	\arrow[from=1-4, to=1-5]
\end{tikzcd}\]
When we restrict $s_0$ to $M=M\otimes \mbb{Z}_p \hookrightarrow M\otimes \mbb{Q}_p$, this morphism factors through the kernel of the bottom right map, and so we recover a map $M\rightarrow H$. In other words, $s_0$ lies in $\Hom(M, H)$, and hence we obtain a canonical isomorphism from the push-out property
\[ \mc{E}_{s_0} \xrightarrow{\sim} \mc{E} \]
compatible with the sections. 

The assignment $(\mc{E}, s) \mapsto s_0$ is well-defined and gives an inverse to the map in the statement of the theorem. That the map is compatible with the group structures is immediate by comparing the push-outs in the definition of $\mc{E}_\varphi$ and of the Baer sum of extensions. \end{proof}

\begin{lemma}\label{lemma.rig-kernel}
The kernel of the induced map
\[ \Hom(M, H)(R) \rightarrow \Ext(M \otimes (\mbb{Q}_p/\mbb{Z}_p), H)(R)	 \]
obtained by forgetting the rigidification is the image of $\Hom(M \otimes \mbb{Q}_p, H)(R)$.
\end{lemma}
\begin{proof}
Indeed, by the push-out property any element of $\Hom(M \otimes \mbb{Q}_p, H)(R)$ restricting to a given $\varphi \in \Hom(M, H)(R)$ gives rise to an isomorphism of $\mc{E}_\varphi$ with the trivial extension, and vice versa. 
\end{proof}

In our main case of interest, we can also understand the image:
\begin{theorem}\label{theorem.canonical-rigidifications}
Suppose $R/p$ is Artinian local with residue field $\kappa$. Then a rigidification of $\mc{E} \in \Ext(M \otimes (\mbb{Q}_p/\mbb{Z}_p), H)$ is equivalent to a splitting of $\mc{E}_\kappa$ in the isogeny category of $p$-divisible groups over $\kappa$. In particular, 
\begin{enumerate}
\item If the residue field $\kappa$ is algebraically closed, then any extension of an \'{e}tale $p$-divisible group by $H$ can be rigidified.
\item If the residue field $\kappa$ is perfect, then to give a rigidification it is equivalent to give a splitting in the isogeny category of $p$-divisible groups over $\kappa$ of the induced extension of $\mbb{Q}_p/\mbb{Z}_p$ by $H_\kappa^\et$. In particular, if $H_\kappa$ is connected, then any extension of an \'{e}tale $p$-divisible group by $H$ can be uniquely rigidified. 
\end{enumerate}
\end{theorem}
\begin{proof}
The first part of the theorem follows from \cref{remark.isogeny-universal-cover}. 
Suppose given an extension $\mc{E}$ as in (1) or (2). Then to split $\tilde{\mc{E}}$ it is equivalent to split $\mc{E}_\kappa$ up to isogeny. In the first case, the category of $p$-divisible groups up to isogeny over $\kappa$ is semi-simple, so it is always split. In the second case, the category may not be semi-simple but the slope decomposition still descends to $\kappa$, so that a splitting occurs purely in the slope zero (\'{e}tale) part. 
\end{proof}
\begin{remark} The $p$-divisible group of the Tate curve over $\mbb{F}_p(\!(q)\!)$, an extension of $\mbb{Q}_p/\mbb{Z}_p$ by $\mu_{p^\infty}$ (see \cref{example.tate-curve}), shows that the assumption that the residue field is perfect in (2) cannot be removed. 	
\end{remark}

\begin{example}
Consider the projection map
\[ \tau\colon\Gal(\overline{\mbb{Q}}_p/\mbb{Q}_p)\rightarrow \Gal(\overline{\mbb{F}}_p/\mbb{F}_p) = \widehat{\mbb{Z}}\rightarrow \mbb{Z}_p.\] 
The Galois representation $\begin{bmatrix} 1 & \tau \\ 0 & 1 \end{bmatrix}$ is the Tate module of a non-trivial extension of $\mbb{Q}_p/\mbb{Z}_p$ by itself over $\mbb{Z}_p$ that cannot be rigidified even if we allow passage to arbitrary finite extensions. 
\end{example}

\begin{example}\label{example.connected-etale}
Suppose (for simplicity) that $\kappa$ is algebraically closed. If $T_p G(\mc{O}_K)=0$, then $T_p G(\kappa) \subset \tilde{G}(\kappa)=\tilde{G}(\mc{O}_K)$ is identified via projection to the first coordinate with the $\mbb{Z}_p$-module $M$ of elements in $G^\wedge(\mc{O}_K)$ that are $p$-divisible in $G(\mc{O}_K)$. The connected-\'{e}tale sequence
\[ 0 \rightarrow G^\wedge \rightarrow G \rightarrow G^\et \rightarrow 0 \]
induces an isomorphism of $M \otimes (\mbb{Q}_p/\mbb{Z}_p)$ with $G^\et$, and this is the extension of $G^\wedge$ determined by the map $M \hookrightarrow G^\wedge$. 
\end{example}

\begin{remark}\label{remark.rig-BKF}
The category of rigidified Breuil--Kisin--Fargues modules (see \cite{anschutz:rigidBKFmodules}, \cite[\S4]{bhatt-morrow-scholze:integral-p-adic-ht} provides a natural category of cohomological motives over $\mbb{C}_p$ (for example, for a smooth proper formal scheme over $\mc{O}_{\mbb{C}_p}$ with torsion-free crystalline cohomology, there is a rigidified Breuil--Kisin--Fargues module in each cohomological degree that recovers all other $p$-adic cohomology theories and their comparisons). The category of $p$-divisible groups over $\mc{O}_{\mbb{C}_p}$ is equivalent to the full sub-category of the category of Breuil--Kisin--Fargues modules with slopes in $[0,1]$, but if $H_{\overline{\mbb{F}}_p}$ is not isoclinic then there is no canonical choice of a rigidification for $H$ --- a rigidification here amounts to the choice of an isogeny $H_{\overline{\mbb{F}}_p} \times_{\overline{\mbb{F}}_p} \mc{O}_{\mbb{C}_p}/p \rightarrow H_{\mc{O}_{\mbb{C}_p}/p}$ inducing the identity modulo $\mf{m}_{\mbb{C}_p}$. If $H$ is connected and equipped with a rigidification, then a rigidified extension as considered in this section is exactly an extension in the category of rigidified Breuil--Kisin--Fargues modules. The natural category analogous to $1$-motives here is the category of rigidified Breuil--Kisin--Fargues modules with slopes in $[0,1]$, and \cref{theorem.canonical-rigidifications}-(2) expresses the fact that the choice of a rational structure over a discretely valued subfield induces a canonical rigidification.
\end{remark}

\section{Proof of main theorem}
\label{s.proof}
In this section we construct the map $\overline{I}$ and prove \cref{theorem.main}.

\subsection{Construction of the  map $\overline{I}$.}
\label{ss.construction-of-map}
We construct $\overline{I}$ as the composition of the following two homomorphisms:
\begin{enumerate}
\item The homomorphism $\mc{E}\colon G(\mc{O}_K) \rightarrow \Ext(\mbb{Q}_p/\mbb{Z}_p, G[p^\infty])$, $x \mapsto \mc{E}_x[p^\infty]$, where
\[ \mc{E}_x\colon 0 \rightarrow G \rightarrow G_x \rightarrow \mbb{Q}_p/\mbb{Z}_p \rightarrow 0 \]
is the extension attached to $\varphi\colon \mbb{Z} \rightarrow G$,  $1 \mapsto x$ by the construction of \cref{s.rig-ext}. 
\item The homomorphism $\psi\colon \Ext(\mbb{Q}_p/\mbb{Z}_p, G[p^\infty])\rightarrow (\Lie G \otimes_{\mc{O}_K} C(1))/T_p G$ sending 
\[ 0 \rightarrow G[p^\infty] \rightarrow H \rightarrow \mbb{Q}_p/\mbb{Z}_p \rightarrow 0 \]
to the image of $1 \in \mbb{Z}_p$ under the right vertical arrow of the diagram induced by applying the canonical splitting of the Hodge--Tate filtration (\cref{eq.tate-splitting}) in the left two terms, 
% https://q.uiver.app/?q=WzAsMTAsWzAsMCwiMCJdLFsxLDAsIlRfcCBHIl0sWzIsMCwiVF9wIEgiXSxbMywwLCJcXG1iYntafV9wIl0sWzQsMCwiMCJdLFswLDEsIjAiXSxbMSwxLCJcXExpZSBHKDEpIl0sWzIsMSwiXFxMaWUgSCgxKSJdLFszLDEsIlxcTGllIEcoMSkgLyBUX3AgRyJdLFs0LDEsIjAiXSxbMyw4XSxbMSw2XSxbMiw3XSxbNSw2XSxbMCwxXSxbMSwyXSxbNiw3LCI9IiwxLHsic3R5bGUiOnsiYm9keSI6eyJuYW1lIjoibm9uZSJ9LCJoZWFkIjp7Im5hbWUiOiJub25lIn19fV0sWzIsM10sWzcsOF0sWzMsNF0sWzgsOV1d
\[\begin{tikzcd}[column sep = 1em]
	0 & {T_p G} & {T_p H} & {\mbb{Z}_p} & 0 \\
	0 & {\Lie G \otimes_{\mc{O}_K} C(1)} & {\Lie G \otimes_{\mc{O}_K} C(1)} & {(\Lie G \otimes_{\mc{O}_K} C(1)) / T_p G} & {}
	\arrow[from=1-4, to=2-4]
	\arrow[from=1-2, to=2-2]
	\arrow[from=1-3, to=2-3]
	\arrow[from=2-1, to=2-2]
	\arrow[from=1-1, to=1-2]
	\arrow[from=1-2, to=1-3]
	\arrow[from=2-2, to=2-3, equal]
	\arrow[from=1-3, to=1-4]
	\arrow[from=2-3, to=2-4]
	\arrow[from=1-4, to=1-5]
	%\arrow[from=2-4, to=2-5]
\end{tikzcd}\]
\end{enumerate}

\begin{remark}\label{remark.canonical-root}
Note that the map $\mbb{Z}[1/p] \rightarrow G_x$ extending $-\varphi$ in the push-out construction of $G_x$ gives rise to a canonical system of $p$-power roots of $-x$ in $G_x(\mc{O}_K)$ via the images of $1/p^n$. We compile these as an element $\widetilde{-x}_\can$ of $\widetilde{G}_x(\mc{O}_K)$ lifting $-x \in G(\mc{O}_K) \subset G_x(\mc{O}_K)$ and projecting to $1 \in \mbb{Q}_p= \widetilde{\mbb{Q}_p/\mbb{Z}_p}=\Hom(\mbb{Z}[1/p], \mbb{Q}_p/\mbb{Z}_p)$. In the explicit coordinates of \cref{eq.carrying-law},  \[\widetilde{-x}_\can :=((-x,0), (0,1/p), (0,1/p^2),\dots).\] 
When $G$ is a $p$-divisible group, $\widetilde{-x}_\can=s_\varphi(1)$ for $s_\varphi$ the canonical rigidification from \cref{s.rig-ext}.
We will use the element $\widetilde{-x}_\can$ in our comparison with other constructions in \cref{s.integration-universal-cover-other-constructions}. 
\end{remark}

\subsection{Proof of \cref{theorem.main}}
In this section, we prove \cref{theorem.main}. 
To begin, we establish two lemmas concerning the maps $\psi$ and $\mc{E}$ defined in \cref{ss.construction-of-map}

\begin{lemma}If $\kappa$ is algebraically closed, then the map $\psi$ is injective.	
\end{lemma}
\begin{proof}
Suppose we have an extension such that $1 \mapsto 0$. That means there a pre-image $v$ of $1$ in $T_p H$ such that $v$ maps to zero in $\Lie H(1)$. The $\mbb{Q}_p$-span $M$ of the Galois orbit of $v$ is thus contained in the kernel	of this map, so $M \otimes C \subset \omega_{G[p^\infty]^\vee} \otimes_{\mc{O}_K} C \subset T_p H \otimes C$. Thus $M$ is of Hodge--Tate weight zero and crystalline (as a subrepresentation of $V_p H$), so the Galois action is trivial since $\kappa$ is algebraically closed. Thus we obtain a splitting
\[ v \in M \subseteq T_p H(K)=\Hom(\mbb{Q}_p/\mbb{Z}_p, H[p^\infty]). \]
\end{proof}

\begin{lemma} If $\kappa$ is algebraically closed, then $\mr{Ker}(\mc{E}) = G(\mc{O}_K)^{\pdiv}.$
\end{lemma}
\begin{proof} We first observe that if $\kappa$ is algebraically closed, then by \cref{theorem.canonical-rigidifications}, $\mc{E}_{|_{G[p^\infty](\mc{O}_K)}}$ is surjective. Then by \cref{lemma.rig-kernel} it induces
\begin{equation}\label{eq.lemma-ext-quot} G[p^\infty](\mc{O}_K) / G[p^\infty](\mc{O}_K)^{\pdiv} = \Ext(\mbb{Q}_p/\mbb{Z}_p, G[p^\infty]). \end{equation}
Since $G[p^\infty](\mc{O}_K)^{\pdiv}$ is a divisible $\mbb{Z}$-module (it is a $\mbb{Z}_p$-module and $p$-divisible), it is injective and thus a direct summand of the $\mbb{Z}$-module $G[p^\infty](\mc{O}_K)$. Thus \cref{eq.lemma-ext-quot} implies $\Ext(\mbb{Q}_p/\mbb{Z}_p, G[p^\infty])$ has no non-zero $p$-divisible elements, so we conclude $G(\mc{O}_K)^{\pdiv}$ is contained in the kernel of $\mc{E}$. On the other hand, since every element of $G(\kappa)$ is $p$-divisible thus admits a lift to a $p$-divisible element of $G(\mc{O}_K)$ by \cref{example.fontaine-p-adic}, we have the factorization as an amalgamated sum
\[ G(\mc{O}_K) = G[p^\infty](\mc{O}_K) \sqcup_{G[p^\infty](\mc{O}_K)^{\pdiv}} G(\mc{O}_K)^{\pdiv}. \]
Combined with \cref{eq.lemma-ext-quot}, we conclude the kernel is identically $G(\mc{O}_K)^{\pdiv}$. 
\end{proof}

Thus, for a general $K$, we find the kernel of $\overline{I}$ is $G(\mc{O}_{(K^\ur)^\wedge})^{\pdiv} \cap G(\mc{O}_K).$ The claim about the kernel in \cref{theorem.main} then follows from
\begin{lemma}\label{lemma.p-div-field-of-def} If $T_p G(K^\ur)=0$, then  
\[  G(\mc{O}_{(K^\ur)^\wedge})^{\pdiv} \cap G(\mc{O}_{K}) = G(\mc{O}_K)^{\pdiv} \]
\end{lemma}
\begin{proof}
Since $T_p G(K^\ur)=T_p G((K^\ur)^\wedge)=0$, any $p$-divisible element is uniquely $p$-divisible. Thus, by considering the Galois action, we find that if 
\[ x \in G(\mc{O}_K) \cap G(\mc{O}_{(K^\ur)^\wedge})^{\pdiv} \]
 then $x^{1/p^n}$ is also in $G(\mc{O}_K)$ for any $n$ so $x \in G(\mc{O}_K)^{\pdiv}$. 
\end{proof}

It remains to show the crystalline characterization of the image in \cref{theorem.main}. On the one hand, any point in the image of $G[p^\infty](\mc{O}_K)$ satisfies $(a)-(c)$ because the Tate module of any $p$-divisible group is crystalline and from the construction the extensions are rigidified. Conversely, suppose given $y$ satisfying $(a)-(c)$. Then the lattice $T_y$ in the crystalline representation $V_y$ has Hodge--Tate weights zero and thus comes from a $p$-divisible group $H$. By full-faithfulness of the Tate module, the extension 
\[ 0 \rightarrow T_p G[p^\infty] \rightarrow T_p H \rightarrow \mbb{Z}_p \rightarrow 0 \]
comes from a diagram of $p$-divisible groups
\[ 0 \rightarrow G[p^\infty] \rightarrow H \rightarrow \mbb{Q}_p/\mbb{Z}_p \rightarrow 0 \]
which is an extension.
It admits a rigidification by condition (c), so comes from a point in $G[p^\infty](\mc{O}_K)$, thus we find $y$ is in the image of $\overline{I}$. 
This concludes the proof of \cref{theorem.main}.

\section{Integration on the universal cover and other constructions}\label{s.integration-universal-cover-other-constructions}

For $G$ a $p$-divisible group or a $p$-adic formal semi-abelian scheme over $\mc{O}_K$, in this section we write $\overline{I}_G$ for the integration map defined in \cref{ss.construction-of-map} to emphasize the dependence on $G$. The integration map $\overline{I}_G$ is induced by a homomorphism 
\begin{equation}\label{eq.fontaine-integral} I_G\colon \tilde{G}(\mc{O}_C) \times_{G(\mc{O}_C)} G(\mc{O}_K) \rightarrow \Lie G \otimes_{\mc{O}_K} C(1) \end{equation}
which we define now: recall from \cref{remark.canonical-root} that, given $x \in G(\mc{O}_K)$ we have a canonical compatible systems of $p$-power roots of $-x$, $\widetilde{-x}_\can \in \tilde{G}_x(\mc{O}_K)$. Given a compatible system of $p$-power roots of $x$, $\tilde{x} \in \tilde{G}(\mc{O}_C)$, the element $\tilde{x} + (\widetilde{-x}_\can)$ lies in $T_p G_x (\mc{O}_C)$, and we define $I(\tilde{x})$ to be the image of $\tilde{x} + (\widetilde{-x}_\can)$ in $\Lie G \otimes_{\mc{O}_K} C(1)=\Lie G_x \otimes_{\mc{O}_K} C(1)$ under the canonical splitting of the Hodge--Tate filtration. It is immediate from the construction that this lifts $\overline{I}_G$. 

In \cref{ss.fontaine-integration} and \cref{ss.deRham-comparison-construction}, we explain two other ways to construct the map $I_G$. The first is via Fontaine integration as in \cite{imz}, while the second passes through the crystalline incarnation of the universal cover of a $p$-divisible group as in \cite{scholze-weinstein:p-div}. That all three constructions agree is a consequence of the following uniqueness result:

\begin{theorem}\label{theorem.uniqueness}
Let $\mc{C}$ be the full subcategory of group-valued functors on $\mr{Nilp}_{\mc{O}_K}$ consisting of functors represented by $p$-divisible groups over $\mc{O}_K$ or by $p$-adic formal abelian schemes over $\mc{O}_K$. The natural transformation of functors from $\mc{C}$ to abelian groups
\[ G \mapsto I_G\colon \tilde{G}(\mc{O}_C)\times_{G(\mc{O}_C)}G(\mc{O}_{K})\rightarrow \Lie G\otimes_{\mc{O}_K} C(1) \]
is the unique natural transformation that is Galois equivariant and agrees with the canonical splitting of the Hodge--Tate filtration on 
\[ T_pG(\mc{O}_C) \subset \tilde{G}(\mc{O}_C)\times_{G(\mc{O}_C)}G(\mc{O}_{K}).\]
\end{theorem}

\begin{proof}
We first verify that $I$ satisfies these properties. The Galois equivariance is immediate; to show it agrees with the canonical splitting of the Hodge--Tate filtration on $T_p G(\mc{O}_C)$, note that if $x \in G[p^n](\mc{O}_K)$ then $\widetilde{-x}_\can \in V_p G_x (\mc{O}_K)$ thus its image under the canonical splitting of the Hodge--Tate filtration is zero (because zero is the only Galois invariant vector in $\Lie G \otimes_{\mc{O}_K} C(1)$ by \cite[Theorem 2]{tate:p-divisible-groups}).

Suppose now given another natural transformation $I'$ satisfiying these properties and let $\tilde{x}=(x,x_1,\dots)\in \tilde{G}(\mc{O}_C)\times_{G(\mc{O}_C)}G(\mc{O}_{K})$. Consider the extension
\[ \mc{E}_{x}\colon 0 \rightarrow G[p^{\infty}] \rightarrow G_{x}[p^{\infty}]\rightarrow \mbb{Q}_p/\mbb{Z}_p \rightarrow 0\]
and the induced sequence of universal covers
\[0 \rightarrow \widetilde{G[p^{\infty}]} \rightarrow \widetilde{G_{x}[p^{\infty}]} \rightarrow \mbb{Q}_p \rightarrow 0.\] 
Again by Galois equivariance, $I'_{G_x}(\widetilde{-x}_\can)=0$ and $I_{G_x}(\widetilde{-x}_\can)=0$. Thus 
\[ I'_G(\tilde{x})=I'_{G_x}(\tilde{x})=I'_{G_x}(\tilde{x}+ (\widetilde{-x}_\can))=I_{G_x}(\tilde{x}+(\widetilde{-x}_\can))=I_{G_x}(\tilde{x})=I_{G}(\tilde{x}) \]
where we have used functoriality and the inclusions $\tilde{G} \subset \tilde{G}_x$ and $\Lie G \otimes_{\mc{O}_K} C(1) = \Lie G_x \otimes_{\mc{O}_K} C(1)$ to make sense of the first and last equality, while the middle equality follows since both $I$ and $I'$ agree with the canonical splitting of the Hodge--Tate filtration on $T_p G_x(\mc{O}_C)$ and 
$ \tilde{x}+(\widetilde{-x}_\can) \in T_p G_x(\mc{O}_C)$. 
\end{proof}

\subsection{Construction via Fontaine integration}\label{ss.fontaine-integration}
For $G$ an abelian scheme or $p$-divisible group over $\mc{O}_K$, in \cite{fontaine:integral}, Fontaine gave an explicit construction of the canonical splitting of the Hodge--Tate filtration via ``integration" along elements of $T_p G$. Concretely, when $G$ is an abelian scheme, given  $(x_i) \in T_p G(\mc{O}_{\overline{K}})$ we may pullback any differential $\omega$ on $G$ to obtain 
\[ (x_i^* \omega) \in V_p \Omega_{\mc{O}_{\overline{K}}/\mc{O}_K}=C(1), \]
where the last equality follows from \cite[Th\'eor\`eme $1'$]{fontaine:integral}. 
For abelian schemes, it was observed in \cite{imz} that this definition extends naturally to a map 
\begin{equation}\label{eq.fontaine-integral} I\colon \tilde{G}(\mc{O}_C) \times_{G(\mc{O}_C)} G(\mc{O}_K) \rightarrow \Lie G \otimes_{\mc{O}_K} C(1) \end{equation}
In fact, just as with the extension to $p$-divisible groups in \cite[\S5]{fontaine:integral}, the definition works also for any $p$-adic formal semi-abelian scheme $G/\Spf \mc{O}_K$:~given an element $(x_i)\in \tilde{G}(\mc{O}_C) \times_{G(\mc{O}_C)} G(\mc{O}_K)$, each $x_i$ gives a map $\Spf \mc{O}_{K_i} \rightarrow G$ where $K_i/K$ is a finite extension, thus there is no problem with the pullback of differentials in the formation of the Fontaine integral (the a priori ``issue" we need to circumvent is that one cannot run the argument with arbitrary $\mc{O}_C$-points, but for a $p$-adic formal semi-abelian scheme there is no such thing as an $\mc{O}_{\overline{K}}$-point since $\mc{O}_{\overline{K}}$ is not $p$-adically complete).

The Galois equivariance is immediate from the construction, and the agreement with the canonical splitting of the Hodge--Tate filtration is established in \cite[\S5]{fontaine:integral} (for $p$-divisible groups, which suffices by functoriality). 

\subsection{Construction via crystalline incarnation of the universal cover}\label{ss.deRham-comparison-construction}
\newcommand{\Fil}{\mr{Fil}}
\newcommand{\Gr}{\mr{Gr}}

We first restrict to the case that $G$ is a $p$-divisible group over $\mc{O}_K$. Let $B^+_\crys \subseteq B^+_\dR$, $B_\crys \subseteq B_\dR$ denote the usual Fontaine period rings for $K$, $\theta\colon B^+_\dR \rightarrow C$ the usual map, and $\Fil^iB_\dR=t^i B^+_\dR \subseteq B_\dR$ for $t$ any generator of $\ker \theta$. By Hensel's lemma, $\theta$ is a map of $K$-algebras. 

We write $D$ for the covariant isocrystal of $G$, which is the Dieudonn\'{e} module of $G_\kappa$ with $p$ inverted and Frobenius divided by $p$ (so, e.g., if $G=\mbb{Q}_p/\mbb{Z}_p$, then the Frobenius is $1$, while if $G=\mbb{G}_m$ it is $1/p$). We write $T=T_p G(\mc{O}_C)$ and $V=T[1/p]$. Evaluation of Dieudonn\'{e} crystals on $B_\crys$ induces (by \cite[Corollary 5.1.2]{scholze-weinstein:p-div}) an identification of $\widetilde{G}(\mc{O}_C)$ with $(D \otimes B^+_\crys)^{\varphi=1}$ (note that this differs from the $\varphi=p$ in \cite{scholze-weinstein:p-div} because in our normalization we have divided the Frobenius on $D$ by $p$), and the map
\[ \qlog\colon \widetilde{G}(\mc{O}_C) = (D \otimes B^+_\crys)^{\varphi=1} \rightarrow D \otimes B^+_\dR \xrightarrow{\theta} D_C \]
is the quasi-logarithm of \cite[Definition 3.2.3]{scholze-weinstein:p-div}. We write $\widetilde{\qlog}$ for the first arrow
\[ \widetilde{G}(\mc{O}_C) \rightarrow D\otimes B^+_\dR. \]

\begin{remark}Alternatively, if we interpret the universal cover as the global sections of the corresponding vector bundle on the Fargues--Fontaine curve, then $\qlog$ is the restriction of global sections to the canonical point $\infty_C$ while $\widetilde{\qlog}$ is the restriction to a formal neighborhood of $\infty_C$.
\end{remark}

As in \cite[Lemma 3.2.5]{scholze-weinstein:p-div}, the logarithm 
\[ \log\colon \tilde{G}(\mc{O}_C) \rightarrow G(\mc{O}_C) \rightarrow (\Lie G)_C \]
can be realized as the composition of $\qlog$ with the projection 
\[ D_C \rightarrow \Lie G_C = \Gr^{-1}(D_K)_C = (D_K/\Fil^{0} D_K)_C,\] 
where $\Fil^{0}D_K=\omega_{G^\vee}$ is the non-trivial step in the Hodge filtration on $D_K$ with quotient $\Lie G=\Gr^{-1}(D_K)=D_K/\Fil^{0}D_K$. We thus write 
\[ \widetilde{\log}: \tilde{G}(\mc{O}_C) \times_{G(\mc{O}_C)} G(\mc{O}_K) \rightarrow \Lie G_{B^+_\dR} \]
for the lift of $\log$ given by composing $\widetilde{\qlog}$ with the projection 
\[ D_{B^+_\dR} \rightarrow \Lie G_{B^+_\dR}=(D_K/\Fil^{0} D_K)_{B^+_\dR}.\]

On the other hand, we also have the constant lift $\log \otimes_K B^+_\dR$ coming from the section $K\rightarrow B^+_\dR$ of $\theta$. The difference 
\[ \widetilde{\log} - \log \otimes_K B^+_\dR \]
is a Galois equivariant homomorphism 
\[ f\colon \tilde{G}(\mc{O}_C) \times_{G(\mc{O}_C)} G(\mc{O}_K) \rightarrow \Lie G \otimes_K \Fil^1 B^+_\dR  \]
and we write $\overline{f}$ for the map obtained by quotienting by $\Lie G \otimes_K \Fil^2 B^+_\dR$
\[ \overline{f}\colon \tilde{G}(\mc{O}_C) \times_{G(\mc{O}_C)} G(\mc{O}_K) \rightarrow \Lie G \otimes_{\mc{O}_K} C(1). \]

We claim that $\overline{f}$ agrees on the Tate module $T$ with the canonical splitting of the Hodge--Tate filtration. To see this, first note that $\log$ is identically zero on $T$, so that $\overline{f}_{|T}$ is just the reduction of $\widetilde{\log}$. We need to check that this is the canonical splitting of the Hodge--Tate filtration. To that end, we note that the inclusion
\[ T \hookrightarrow \tilde{G}(\mc{O}_C)=(D \otimes B^+_\crys)^{\varphi=1} \]
is the restriction to $T$ of the crystalline comparison isomorphism
\[ V \otimes B_\crys = D \otimes B_\crys. \]
As with any de Rham representation, the $i$th component of the Hodge--Tate grading 
\[ V \otimes C = \bigoplus_i \Gr^{-i}D_K \otimes_K C(i) \]
is induced by first multiplying $\Fil^{-i} D_K$  with $\Fil^{i}{B}_\dR$ to land in $\Fil^0 (D_{B_\dR}) = V_{ B^+_\dR}$ and then projecting to $V_C$. 
The identification then follows from the following commutative diagram, whose first two rows illustrate the canonical splitting of the Hodge--Tate filtration in the top row and whose bottom row illustrates the definition of $\widetilde{\log}$.
In reading the diagram and comparing with the above, it may be helpful to keep in mind the identifications:
\begin{multline*} \Fil^0(D_K \otimes B_\dR)=V \otimes \Fil^0 B_\dR = V \otimes B^+_\dR,\; \Fil^{-1} D_K=D_K,\\ \Gr^0(D_K)=\Fil^0(D_K)=\omega_{G^\vee} \otimes K, \textrm{ and } \Gr^{-1} D_K=\Lie G \otimes K\end{multline*}% https://q.uiver.app/?q=WzAsMTIsWzEsMCwiViBcXG90aW1lcyBDIl0sWzIsMSwiKFxcRmlsXnstMX1EX0spXFxvdGltZXMgXFxGaWxeMSBCX1xcZFIiXSxbMiwwLCIoXFxHcl57LTF9RF9LKVxcb3RpbWVzIEMoMSkiXSxbMCwwLCIoXFxGaWxeMCBEX0spXFxvdGltZXMgQyJdLFsxLDMsIihcXEZpbF57LTF9IERfSykgXFxvdGltZXMgQl9cXGRSIl0sWzAsMywiKFxcRmlsXjBEX0spXFxvdGltZXMgQl9cXGRSIl0sWzIsMywiKFxcR3Jeey0xfURfSykgXFxvdGltZXMgQl9cXGRSIl0sWzEsNV0sWzIsMiwiKFxcR3Jeey0xfSBEX0spIFxcb3RpbWVzIFxcRmlsXjEgQl9cXGRSIl0sWzEsMSwiXFxGaWxeMChEX0sgXFxvdGltZXMgQl9cXGRSKSJdLFsxLDIsIlQiXSxbMCwxLCIoXFxGaWxeMCBEX0spXFxvdGltZXNcXEZpbF4wIEJfe1xcZFJ9Il0sWzEsMl0sWzMsMCwiIiwyLHsic3R5bGUiOnsidGFpbCI6eyJuYW1lIjoiaG9vayIsInNpZGUiOiJ0b3AifX19XSxbNSw0LCIiLDAseyJzdHlsZSI6eyJ0YWlsIjp7Im5hbWUiOiJob29rIiwic2lkZSI6InRvcCJ9fX1dLFs0LDYsIiIsMCx7InN0eWxlIjp7ImhlYWQiOnsibmFtZSI6ImVwaSJ9fX1dLFs4LDYsIiIsMCx7InN0eWxlIjp7InRhaWwiOnsibmFtZSI6Imhvb2siLCJzaWRlIjoiYm90dG9tIn19fV0sWzEsOCwiIiwxLHsic3R5bGUiOnsiaGVhZCI6eyJuYW1lIjoiZXBpIn19fV0sWzgsMiwiIiwxLHsiY3VydmUiOjUsInN0eWxlIjp7ImhlYWQiOnsibmFtZSI6ImVwaSJ9fX1dLFswLDIsIiIsMix7Im9mZnNldCI6LTEsInN0eWxlIjp7ImhlYWQiOnsibmFtZSI6ImVwaSJ9fX1dLFsxLDksIiIsMCx7InN0eWxlIjp7InRhaWwiOnsibmFtZSI6Imhvb2siLCJzaWRlIjoiYm90dG9tIn19fV0sWzksMCwiIiwwLHsic3R5bGUiOnsiaGVhZCI6eyJuYW1lIjoiZXBpIn19fV0sWzEwLDQsIiIsMSx7InN0eWxlIjp7InRhaWwiOnsibmFtZSI6Imhvb2siLCJzaWRlIjoiYm90dG9tIn19fV0sWzEwLDksIiIsMSx7InN0eWxlIjp7InRhaWwiOnsibmFtZSI6Imhvb2siLCJzaWRlIjoiYm90dG9tIn19fV0sWzEwLDgsIlxcd2lkZXRpbGRle1xcbG9nfSJdLFsyLDAsIiIsMCx7Im9mZnNldCI6LTF9XSxbMTEsOSwiIiwxLHsic3R5bGUiOnsidGFpbCI6eyJuYW1lIjoiaG9vayIsInNpZGUiOiJ0b3AifX19XSxbMTEsM11d
\[\begin{tikzcd}
	{(\Fil^0 D_K)\otimes C} & {V \otimes C} & {(\Gr^{-1}D_K)\otimes C(1)} \\
	{(\Fil^0 D_K)\otimes\Fil^0 B_{\dR}} & {\Fil^0(D_K \otimes B_\dR)} & {(\Fil^{-1}D_K)\otimes \Fil^1 B_\dR} \\
	& T & {(\Gr^{-1} D_K) \otimes \Fil^1 B_\dR} \\
	{(\Fil^0D_K)\otimes B_\dR} & {(\Fil^{-1} D_K) \otimes B_\dR} & {(\Gr^{-1}D_K) \otimes B_\dR} \\
	\\
	& {}
	\arrow[from=2-3, to=1-3]
	\arrow[hook, from=1-1, to=1-2]
	\arrow[hook, from=4-1, to=4-2]
	\arrow[two heads, from=4-2, to=4-3]
	\arrow[hook', from=3-3, to=4-3]
	\arrow[two heads, from=2-3, to=3-3]
	\arrow[bend right=90, two heads, from=3-3, to=1-3]
	\arrow[shift left=1.5, two heads, from=1-2, to=1-3]
	\arrow[hook', from=2-3, to=2-2]
	\arrow[two heads, from=2-2, to=1-2]
	\arrow[hook', from=3-2, to=4-2]
	\arrow[hook', from=3-2, to=2-2]
	\arrow["{\widetilde{\log}}", from=3-2, to=3-3]
	\arrow[hook', shift left=1.5, from=1-3, to=1-2]
	\arrow[hook, from=2-1, to=2-2]
	\arrow[from=2-1, to=1-1]
\end{tikzcd}\]

To extend this definition  to $p$-adic formal semi-abelian schemes $G/\Spf \mc{O}_K$, we may assume $\kappa$ is algebraically closed then use the decomposition 
\[ \tilde{G}(\mc{O}_C) = \widetilde{G^\wedge} (\mc{O}_C) \times \tilde{G}(\kappa) \]
and projection onto $\widetilde{G^\wedge}$ followed by $\overline{f}$. The properties verified above and \cref{theorem.uniqueness} show that this construction recovers $I$. 

\begin{example}When $G=\mu_{p^\infty}$, one obtains the explicit formula
\[ I(q,q^{1/p},q^{1/p^2},\ldots)\left(\frac{dt}{t}\right)=[(q,q^{1/p},q^{1/p^2},\ldots)]-q \in C(1)=\ker \theta / (\ker\theta)^2. \]
 where $q \in 1+\mf{m}_{\overline{K}}$, $(q,q^{1/p},q^{1/p^2},\ldots ) \in (1+ m_{C})^\flat$ is a compatible family of $p$-power roots of $q$, and $[\cdot]$ denotes the multiplicative lift  to $W(\mc{O}_{C^\flat})\subset B^+_\dR$. 
 \end{example}

\begin{remark}\label{remark:continuity}
Although $\widetilde{\log}$ and $\log$ are both continuous on 
\[ \bigcup_{\substack{[K':K]<\infty\\ K'\subset \overline{K}}} \tilde{G}(\mc{O}_C) \times_{G(\mc{O}_C)} G(\mc{O}_{K'}) \]
for the Banach-Colmez topology on $\tilde{G}(\mc{O}_C)$, $\log \otimes B^+_\dR/t^2$ is not continuous  because the canonical section $\overline{K} \hookrightarrow B^+_\dR/t^2$ provided by Hensel's lemma is not continuous for the $p$-adic topology induced by $\overline{K} \subset C$. Because of this, $I$ is not continuous on this set (and this gives another way to explain why it does not extend to $\tilde{G}(\mc{O}_C)$).
As in \cite{imz}, this can be rectified at the level of 
\[ \overline{I}\colon \bigcup_{\substack{[K':K]<\infty\\ K'\subset \overline{K}}}G(\mc{O}_{K'}) \rightarrow (\Lie G \otimes_{\mc{O}_K} C(1))/T_p G, \]
by replacing the $p$-adic topology on the left with a stronger topology induced locally by the inclusion of $\overline{K} \subset B^+_\dR/t^2$ (cf.~\cite[Remark 2.4]{imz}).  
This issue is invisible if we only work with $\mc{O}_K$-points where the two topologies agree. 
\end{remark}

\section{Geometric structure of the codomain of $\overline{I}$}\label{s.geometry}
 In this short section, we  discuss the geometric structure of the codomain of the integration map $\overline{I}$. 
Let $G$ be a $p$-divisible group over $\mathcal{O}_K$. Below we work exclusively with the adic generic fibers as in \cite{scholze-weinstein:p-div} of $G$, $\tilde{G}$, and $\Lie G$.

The universal cover $\tilde{G}$ is an effective Banach--Colmez space and admits a natural geometric structure as a diamond \cite[\S15.2]{scholze:p-adic-geometry}. The projection to $G$ realizes $G=\tilde{G}/T_p G$ (where the equality here is as diamonds),  i.e. it gives a profinite \'{e}tale Scholze--Weinstein uniformization of $G$ by an effective Banach--Colmez space. This uniformization is robust both topologically and geometrically.

Although the conjugate uniformization is neither continuous nor geometric, its codomain does have a natural geometric structure:~consider the maps
\[ \Lie G \otimes_{\mc{O}_K} C(1) \twoheadrightarrow (\Lie G \otimes_{\mc{O}_K} C(1))/T_p G \twoheadrightarrow (\Lie G \otimes_{\mc{O}_K} C(1))/V_p G. \]
The space $\Lie G \otimes_{\mc{O}_K} C(1)$ is as an effective Banach--Colmez space over $K$ (a very non-trivially twisted affine space), while the space $(\Lie G \otimes_{\mc{O}_K} C(1))/V_p G$ is a negative Banach--Colmez space over $K$ in the sense of \cite{fargues-scholze:geometrization}. Thus $(\Lie G \otimes_{\mc{O}_K} C(1))/T_p G$ can  be thought of as either a profinite-\'{e}tale quotient of an effective Banach--Colmez space or as an \'{e}tale cover of a negative Banach--Colmez space; in particular, all three spaces in the diagram are naturally diamonds over $K$.  

\begin{question} We thus have an identification of the rigid analytic points (over $K$) of the diamond $G$ with a subset of the rigid analytic points (over $K$) of the diamond $(\Lie G \otimes_{\mc{O}_K} C(1))/T_p G$ cut out by a $p$-adic Hodge theoretic condition. Are there any higher dimensional rigid analytic subdiamonds of $(\Lie G \otimes_{\mc{O}_K} C(1))/T_p G$? If so, can any be distinguished by conditions in relative $p$-adic Hodge theory and matched with rigid analytic subvarieties of $G$ parameterizing rigidified extensions of $\mbb{Q}_p/\mbb{Z}_p$ by $G$? In some related contexts there are interesting answers to these kinds of questions. For example, the non-miniscule open Schubert cells in $B^+_\dR$-affine Grassmannians are diamonds which are not rigid analytic, but their rigid analytic subvarieties still admit a nice description via $p$-adic Hodge theory --- via the Bialynicki-Birula map of \cite[Proposition 3.4.3]{caraiani-scholze:generic}, they are identified with maps to a flag variety satisfying Griffiths transversality (cf.~\cite[\S6]{scholze:p-adic-ht}). Some related questions in the context of moduli spaces of rigidified Breuil--Kisin--Fargues modules (see \cref{remark.rig-BKF} for the connection to the present work) are treated in \cite{howe-klevdal} (see also \cite{howe:j}). 
\end{question}

\bibliographystyle{plain}
\bibliography{refs}

\begin{thebibliography}{10}

\bibitem{anschutz:rigidBKFmodules}
Johannes Ansch\"{u}tz.
\newblock Breuil--{K}isin--{F}argues modules with complex multiplication.
\newblock {\em J. Inst. Math. Jussieu}, 20(6):1855--1904, 2021.

\bibitem{barbieri-viale:1-motives}
Luca Barbieri-Viale.
\newblock On the theory of 1-motives.
\newblock In {\em Algebraic cycles and motives. {V}ol. 1}, volume 343 of {\em
  London Math. Soc. Lecture Note Ser.}, pages 55--101. Cambridge Univ. Press,
  Cambridge, 2007.

\bibitem{bhatt-morrow-scholze:integral-p-adic-ht}
Bhargav Bhatt, Matthew Morrow, and Peter Scholze.
\newblock Integral p-adic hodge theory.
\newblock {\em Publ. Math. de l'IH\`ES}, 128:219--397, 2018.

\bibitem{caraiani-scholze:generic}
Ana Caraiani and Peter Scholze.
\newblock On the generic part of the cohomology of compact unitary {S}himura
  varieties.
\newblock {\em Ann. of Math. (2)}, 186(3):649--766, 2017.

\bibitem{deligne:hodge-III}
Pierre Deligne.
\newblock Th\'{e}orie de {H}odge. {III}.
\newblock {\em Inst. Hautes \'{E}tudes Sci. Publ. Math.}, (44):5--77, 1974.

\bibitem{fargues-scholze:geometrization}
Laurent Fargues and Peter Scholze.
\newblock Geometrization of the local langlands correspondence.
\newblock Preprint.

\bibitem{fontaine:integral}
Jean-Marc Fontaine.
\newblock Formes diff\'{e}rentielles et modules de {T}ate des vari\'{e}t\'{e}s
  ab\'{e}liennes sur les corps locaux.
\newblock {\em Invent. Math.}, 65(3):379--409, 1981/82.

\bibitem{howe:j}
Sean Howe.
\newblock Transcendence of the {H}odge-{T}ate filtration.
\newblock {\em J. Th\'{e}or. Nombres Bordeaux}, 30(2):671--680, 2018.

\bibitem{howe:circle-action}
Sean Howe.
\newblock A unipotent circle action on {$p$}-adic modular forms.
\newblock {\em Trans. Amer. Math. Soc. Ser. B}, 7:186--226, 2020.

\bibitem{howe-klevdal}
Sean Howe and Christian Klevdal.
\newblock In preparation.

\bibitem{imz}
Adrian Iovita, Jackson~S. Morrow, and Alexandru Zaharescu.
\newblock On $p$-adic uniformization of abelian varieties with good reduction.
\newblock {\em Prepint, arXiv:2107.09165 (to appear in \textit{Compositio
  Mathematica})}.

\bibitem{katz:serre-tate}
Nicholas Katz.
\newblock Serre-{T}ate local moduli.
\newblock {\em Lecture {N}otes in {M}athematics}, 868:138--202, 1981.

\bibitem{kisin:cyrstallinereps}
Mark Kisin.
\newblock Crystalline representations and {$F$}-crystals.
\newblock In {\em Algebraic geometry and number theory}, volume 253 of {\em
  Progr. Math.}, pages 459--496. Birkh\"{a}user Boston, Boston, MA, 2006.

\bibitem{scholze:p-adic-ht}
Peter Scholze.
\newblock {$p$}-adic {H}odge theory for rigid-analytic varieties.
\newblock {\em Forum Math. Pi}, 1:e1, 77, 2013.

\bibitem{scholze-weinstein:p-div}
Peter Scholze and Jared Weinstein.
\newblock Moduli of {$p$}-divisible groups.
\newblock {\em Camb. J. Math.}, 1(2):145--237, 2013.

\bibitem{scholze:p-adic-geometry}
Peter Scholze and Jared Weinstein.
\newblock {\em Berkeley Lectures on p-adic Geometry}.
\newblock Princeton University Press, 2020.

\bibitem{tate:p-divisible-groups}
J.~T. Tate.
\newblock {$p$}-divisible groups.
\newblock In {\em Proc. {C}onf. {L}ocal {F}ields ({D}riebergen, 1966)}, pages
  158--183. Springer, Berlin, 1967.

\end{thebibliography}

\end{document}